\documentclass[11pt]{amsart}

\usepackage{amsmath}
\usepackage{amssymb}
\usepackage{amsthm}
\usepackage[all,cmtip]{xy}
\usepackage{color}

\theoremstyle{plain}
\newtheorem{thm}{Theorem}[section]
\newtheorem{prop}[thm]{Proposition}

\theoremstyle{definition}

\newtheorem{eg}[thm]{Example}

\theoremstyle{remark}
\newtheorem{rem}[thm]{Remark}
\newtheorem{claim}[thm]{Claim}

\newcommand{\bC}{\ensuremath{\mathbb{C}}}

\newcommand{\bP}{\ensuremath{\mathbb{P}}}

\newcommand{\bZ}{\ensuremath{\mathbb{Z}}}

\newcommand{\cL}{\ensuremath{\mathcal{L}}}

\newcommand{\cN}{\ensuremath{\mathcal{N}}}
\newcommand{\cO}{\ensuremath{\mathcal{O}}}

\newcommand{\cX}{\ensuremath{\mathcal{X}}}

\DeclareMathOperator{\Pic}{Pic}
\DeclareMathOperator{\Sing}{Sing}

\DeclareMathOperator{\Aut}{Aut}

\DeclareMathOperator{\Image}{Im}

\DeclareMathOperator{\id}{id}

\DeclareMathOperator{\NS}{NS}

\begin{document}

\title
[Non-K\"{a}hler Calabi-Yau manifolds]
{Construction of non-K\"{a}hler Calabi-Yau manifolds by log deformations}

\author{Taro Sano}
\address{Department of Mathematics, Faculty of Science, Kobe University, Kobe, 657-8501, Japan}
\email{tarosano@math.kobe-u.ac.jp}

\maketitle

\begin{abstract} 
	We survey recent construction of non-K\"{a}hler Calabi-Yau manifolds by smoothing SNC varieties obtained by non-trivial isomorphisms of strict Calabi-Yau manifolds. We also give a new example by smoothing an SNC 3-fold which are constructed from automorphisms of an abelian surface.  
		\end{abstract}


\section{Introduction} 

In this note, we survey recent constructions of non-K\"{a}hler Calabi-Yau manifolds, which proceed by smoothing simple normal crossing varieties obtained by gluing pairs of manifolds along strict Calabi-Yau hypersurfaces identified by a non-trivial isomorphism (cf. \cite{Hashimoto:aa}, \cite{MR4406696}). We also give a new example of such a construction, by smoothing simple normal crossing threefolds which are constructed by gluing two threefolds along an abelian surface. 
A compact complex manifold $X$ of dimension $n$ is called a {\it Calabi-Yau manifold} ({\it CY $n$-fold}) if $\omega_X \simeq \cO_X$ and $H^i(X, \cO_X) = 0 = H^0(X, \Omega^i_X)$ for $0 < i < \dim X$. 
In some literature, it is also called a {\it strict Calabi-Yau manifold} to emphasize the condition on the cohomology groups. 

Projective Calabi-Yau manifolds are one of building blocks in the classification of algebraic varieties. 
Note that 1-dimensional Calabi-Yau manifolds are elliptic curves and 2-dimensional Calabi-Yau manifolds are K3 surfaces. 
It is an open problem whether there are only finitely many topological types of projective Calabi-Yau manifolds in a fixed dimension $N \ge 3$. 

On the other hand, it had been known that there are infinitely many topological types of non-K\"{a}hler Calabi-Yau 3-folds (cf. \cite{MR1141199}, \cite{Hashimoto:aa} and references therein). 

%
%
%

Our examples in \cite{MR4406696} are the following. 

\begin{thm}\label{thm:intro}
Let $N \ge 3$ and $m \in \bZ_{>0}$. 

There exists a non-K\"{a}hler Calabi--Yau $N$-fold $X(m)$ with the following properties 
\begin{enumerate}
\item[(i)] The 2nd Betti number is $b_2(X(m)) = \begin{cases}
m+1 & (N=3)\\
m+10 & (N=4) \\
m+2 & (N \ge 5)
\end{cases}$. 
\item[(ii)] The algebraic dimension of $X(m)$ is $N-2$ and $X(m)$ admits a K3 fibration $X(m) \rightarrow T$ to a smooth rational variety $T$. 
\item[(iii)] The Hodge to de Rham spectral sequence on $X(m)$ degenerates at $E_1$. 
\end{enumerate}
\end{thm}

An essential ingredient of the construction is the log deformation theory of normal crossing varieties developed by Kawamata--Namikawa (See also recent results by \cite{Chan:2019vv} and \cite{MR4304077}). 

\begin{rem}
Lee \cite{https://doi.org/10.48550/arxiv.2102.12656} constructed an example of a non-K\"{a}hler Calabi-Yau $4$-fold by smoothing SNC varieties. 
\end{rem}

\begin{rem}
Note that, if $X$ and $Y$ are Calabi--Yau manifolds such that $\dim X, \dim Y >0$, then $X\times Y$ also has a trivial canonical bundle, but it does not satisfy $h^i(\cO_{X \times Y}) =0$ for $0 < i < \dim (X \times Y)$. 
\end{rem}
	
Let us explain the contents of this note. 
In section 2, we explain the results on smoothing of an SNC CY variety and give a brief review of previous examples of non-K\"{a}hler CY 3-folds. In section 3, we give a brief explanation of the examples in \cite{MR4406696}. In section 4, we give further examples of non-K\"{a}hler CY 3-folds obtained by smoothing SNC CY 3-folds which are union of two smooth varieties glued along an abelian surface. We also comment on some open problems. 	
	
\section{Log deformation theory of SNC Calabi--Yau varieties}

In this section, we shall explain log deformation theory. 
We say that a proper SNC $\bC$-scheme $X$ is an {\it SNC Calabi--Yau variety} if $\omega_X \simeq \cO_X$. 
The following is an essential ingredient. A proper SNC $\bC$-scheme is {\it d-semistable} if we have an isomorphism of  sheaves $Ext^1_{\cO_X} (\Omega^1_X, \cO_X) \simeq \cO_D$, where $D:= \Sing X$ is the singular locus of $X$. 

\begin{rem}
Let us recall the definition of d-semistability \cite[Definition 1.13]{MR707162} in a simple case. 
Let $X= X_1 \cup X_2$ be an SNC variety which is a union of two smooth varieties $X_1$ and $X_2$. Let $D:= X_1 \cap X_2$ be the singular locus of $X$. 
Then $X$ is d-semistable if $\cN_{D/X_1} \otimes \cN_{D/X_2} \simeq \cO_D$. It is also equivalent to the existence of a log smooth structure on $X$ (cf. \cite{MR1296351}, \cite{MR1404507}). 
\end{rem}

\begin{thm}\label{thm:smoothingSNC}(\cite{MR1296351}, \cite{Chan:2019vv})
Let $X$ be an SNC Calabi--Yau variety. 
Assume that $X$ is d-semistable. 

Then $X$ admits a deformation $\phi \colon \cX \rightarrow \Delta^1$ over a unit disk $\Delta^1$ such that $\cX$ is smooth and $\cX_t:= \phi^{-1}(t)$ is smooth for $t \neq 0$. 
\end{thm}

\begin{rem}
In the above theorem, if $\dim X \ge 3$, $H^i(X, \cO_X) =0$ for $0 < i < \dim X$ and $X$ is projective, then the general fiber $\cX_t$ is a projective CY manifold by $H^2(X, \cO_X) =0$ and Grothendieck's existence theorem. 
Note that Theorem \ref{thm:smoothingSNC} holds even when $X$ is not projective. 
Then the smoothing $\cX_t$ can be non-projective in general.  
\end{rem}	

\begin{eg}
We can construct an SNC variety $X_0 = X_1 \cup X_2$ by gluing smooth varieties $X_1$ and $X_2$ with smooth divisors $D_i \subset X_i$ for $i=1,2$ with an isomorphism $\phi \colon D_1 \xrightarrow{\simeq} D_2$. 
If $D_i \in |{-}K_{X_i}|$ for $i=1,2$ and they are connected, then we see that $X_0$ satisfies that $\omega_{X_0} \simeq \cO_{X_0}$, thus $X_0$ is an SNC CY. We write $X_0=: X_1\cup^{\phi} X_2$. 

Note that we really need the connectedness assumption. 
Indeed, if $E$ is an elliptic curve and $X_1 = X_2 = \bP^1 \times E$, 
then we can glue $X_1$ and $X_2$ along $\{0, \infty \} \times E$ via 
 the identity $\id \in \Aut (0 \times E)$ and the negation $ \phi_{-1} \in \Aut (\infty \times E)$. Then we see that $X_0:= X_1 \cup X_2$ satisfies that $\omega_{X_0} \not\simeq \cO_{X_0}$ since we compute $H^0(\omega_{X_0}) =0$ by the exact sequence 
 \[
 0 \rightarrow \omega_{X_0} \rightarrow \omega_{X_1}(D_1) \oplus \omega_{X_2} (D_2) \xrightarrow{(r_1, -r_2)} \omega_{D_1} \rightarrow 0, 
 \]
 where $r_1$ is the usual residue map and $r_2$ is a composition of the residue map and $\phi^*$ for $\phi \colon D_1 \xrightarrow{\sim} D_2$. 
\end{eg}

We give a review on non-K\"{a}hler Calabi-Yau 3-folds constructed in the 90's. 

\begin{eg}
We can construct infinitely many topological types of non-K\"{a}hler CY 3-folds as analytic flops of smooth rational curves of large degrees on a projective CY 3-fold. 
Let us briefly recall the construction in the following (cf. \cite{MR720930}, \cite{MR1282199}). 

Let $S=(q = 0) \subset \bP^3$ be a smooth quartic surface defined by $q \in |\cO_{\bP^3}(4)|$ with a smooth rational curve $C_d \subset S$ of degree $d$ which exists by \cite{MR771073}.  
Let $l \in |\cO_{\bP^3}(1)|$ be a linear form and embed $\bP^3$ as a hyperplane $(z_4=0) \subset \bP^4$. 
Let $$V_0:= (q \cdot l + z_4 \cdot f =0) \subset \bP^4$$ be a quintic hypersurface, where $f \in |\cO_{\bP^4}(4)|$ is a general quartic form. 
Then, for a general choice of $l, f$, we see that 
$$\Sing V_0= (q=l=z_4= f=0) \subset \bP^4$$ and they are 16 nodes.  
Note that the family of hyperplanes $(az_4 - b l =0) \subset \bP^4$ for $[a : b] \in \bP^1$ induces a family of quartic surfaces $\{S_{[a : b]} \mid [a : b] \in \bP^1 \}$ such that $S_{[1 : 0]} =S$ and $C_d \subset S$ does not lift to the deformation since $f$ is general (cf. \cite[Lemma 2.6]{MR1282199}). 
Hence we see that the normal bundle $\cN_{C_d/V_0} \simeq \cO_{C_d}(-1)^{\oplus 2}$, that is, $C_d$ is a $(-1,-1)$-curve on $V$. 
This implies that a general quintic 3-fold $V \subset \bP^4$ contains a $(-1,-1)$-curve $C$ of degree $d$. 
By taking an analytic flop $V \dashrightarrow V_d$ of $C$, we obtain a Moishezon CY 3-fold $V_d$ such that $\Pic V_d = \bZ H_d$, where $H_d \subset V_d$ is the strict transform of the hyperplane  $H= \cO_V(1)$. We check that $H_d^3 = 5-d^3$, thus we see that $V_d$ is non-K\"{a}hler when $d \ge 2$. This implies that $\{V_d \mid d \ge 2 \}$ provides infinitely many topological types of non-K\"{a}hler CY 3-folds since the cubic forms on $H^2(V_d, \bZ)$ are different.  
\end{eg}

\begin{eg}
Clemens, Friedman and Reid (\cite{MR720930}, \cite{MR1141199}, \cite{MR909231}) constructed infinitely many topological types of non-K\"{a}hler CY 3-folds as conifold transitions of a general smooth quintic 3-fold $V$ with infinitely many pairwise disjoint $(-1,-1)$-curves $\Gamma_1, \Gamma_2, \cdots$. 
We give a review on the construction in the following. 

Let $V:= (ql + z_4 f =0) \subset \bP^4$ be a quintic 3-fold as above. 
Let $D_i \subset |\cO_{\bP^3}(4)|$ be the divisor which contains the locus of smooth quartic surfaces with smooth rational curves of degrees $i$ (For example, the image of the corresponding flag Hilbert scheme).  
The family of quartic surfaces $\{S_{[a:b]} \mid [a:b] \in \bP^1 \}$ induces a curve $$\gamma= \gamma_{q,l,f} \colon \bP^1 \rightarrow |\cO_{\bP^3}(4)|$$ of degree $5$. 
For a general choice of $q, l, f$, we see that $\gamma$ and $D_i$ intersects (transversely) at (distinct) points which are not on any $D_j (j \neq i)$ and correspond to smooth quartic surfaces. We can pick distinct $[a_i  \colon b_i] \in \bP^1$ so that the corresponding quartic surface $S_{[a_i:b_i]}$ contains a smooth rational curve $\Gamma_i$ of degree $i$. By choosing general $l$ and $f$, we may take $\Gamma_i$ which avoids the nodes of $V$. Hence $\{\Gamma_i \mid i=1,2,\ldots \}$ are pairwise disjoint.  
By smoothing the nodes, we see that a general quintic 3-fold $V$ contains pairwise disjoint  smooth rational curves $\Gamma_i$ of degrees $i$ for $i=1,2, \ldots$.


Let $m \in \bZ_{>0}$ be any positive integer and choose $\Gamma_1, \ldots , \Gamma_m \subset V$. Then we have a birational contraction $\mu \colon V \rightarrow \bar{V}_m$ of $\Gamma_1, \ldots , \Gamma_m$ to a Moishezon space $\bar{V}_m$. 
Then $\bar{V}_m$ has $m$ ordinary double points $p_i := \mu(\Gamma_i)$ and there is a smoothing of $\bar{V}_m$ to a Calabi-Yau 3-fold $Y_m$ by \cite[Corollary 4.7]{MR848512}. 
Then we see that $b_2(Y_m) =0$ and $e(Y_m) = -200-2m$, where $b_2 (Y_m)$ is the 2nd Betti number and $e(Y_m)$ is the topological Euler number. 

We are not sure whether these constructions can be generalized in  higher dimensions.   
\end{eg}	
	
Non-K\"{a}hler CY manifolds including non-strict ones has been studied by various authors (cf. \cite{MR2679581}, \cite{MR3372471},  \cite{MR3784517} and references therein).	
	
\section{Construction of non-K\"{a}hler Calabi-Yau manifolds}\label{section:CYnexamples}	
	
In the construction of our non-K\"{a}hler Calabi-Yau manifolds, the choice of an isomorphism	$\phi \colon D_1 \rightarrow D_2$ to construct an SNC variety is crucial. We give a brief review of the examples in Theorem \ref{thm:intro}. 
	
\begin{eg} 
Let $S:= (s F_1 + t F_2 =0) \subset \bP^2 \times \bP^1$ be a general hypersurface of bi-degree $(3,1)$, where $F_1, F_2 \in |\cO_{\bP^2}(3)|$ are general cubics and $[s:t] \in \bP^1$ are homogeneous coordinates. Then we see that the 1st projection $\pi_1 \colon S \rightarrow \bP^2$ is a blow-up at 9 points $\{p_1, \ldots ,p_9 \} = (F_1 = F_2 =0) \subset \bP^2$. We also see that $\pi_2 \colon S \rightarrow \bP^1$ is an elliptic fibration, thus $S$ is a rational elliptic surface. 


A quadratic transformation $\psi_{123} \colon \bP^2 \dashrightarrow \bP^2$ at $p_1, p_2, p_3$ induces an isomorphism $\phi_{123} \colon S \xrightarrow{\simeq} S_{123}$ to another hypersurface $S_{123} \subset \bP^2 \times \bP^1$ of bi-degree $(3,1)$. We can perform the same operations to construct 
isomorphisms 
\[
S \rightarrow S_{123} \rightarrow S_{456} \rightarrow S_{789}=S' \rightarrow S'_{123} \rightarrow S'_{456} \rightarrow S'_{789}=:S_1. 
\]
Now let $\psi_1: S \rightarrow S_1$ be the composition of these 6 isomorphisms. 
Finally, for any $m \in \bZ_{>0}$, we repeat this process $m$-times to obtain 
\[
\phi_m \colon S \xrightarrow{\psi_1} S_1 \rightarrow \cdots \rightarrow S_{m-1} \xrightarrow{\psi_{m}} S_m. 
\]
An important feature of $\phi_m$ is the following:  

\begin{prop}\cite[Proposition 2.6(iii)]{MR4406696} Let $S, S_m \subset \bP^2 \times \bP^1$ be the $(3,1)$-hypersurfaces as above. 
 Let $H_S:= \mu_S^* \cO_{\bP^2}(1)$ and $H_{S_m}:= \mu_{S_m}^* \cO_{\bP^2}(1)$, where $\mu_S \colon S \rightarrow \bP^2$ and $\mu_{S_m} \colon S_m \rightarrow \bP^2$ are the projections. Let $\phi_m \colon S \rightarrow S_m$ be the isomorphism as above. 

Then the linear system $|H_S + \phi_m^* H_{S_m} +mK_S|$ is ample and free. 
\end{prop}

It is not clear that the isomorphism $\phi_m$ can be realized as an automorphism of $S \subset \bP^2 \times \bP^1$ or not, but this is enough for our purpose. We also use the hypersurface $T$ as follows. 


\begin{prop} \cite[Proposition 2.9]{MR4406696}
Let $n \ge 2$ and let 
$$T:= (s G_1 + t G_2 =0) \subset \bP^1 \times \bP^n$$ be a general hypersurface of bi-degree $(1, n+1)$, where $[s\colon t] \in \bP^1$ are homogeneous coordinates and $G_1, G_2 \in |\cO_{\bP^n}(n+1)|$ are general.	

\item[(i)] Then the projection $T \rightarrow \bP^1$ is a Calabi-Yau fibration and the projection $T \rightarrow \bP^n$ is the blow-up along the subvariety $(G_1 = G_2 =0) \subset \bP^n$. 

\item[(ii)] Let $$D_S:= S \times \bP^n, D_T:= \bP^2 \times T \subset \bP^2 \times \bP^1 \times \bP^n$$ be divisors and $D_{ST}:= D_S \cap D_T$. 
Then $D_{ST}$ is a projective Calabi-Yau manifold and we have an isomorphism $D_{ST} \simeq S \times_{\bP^1} T$. 	
\end{prop}

\noindent{(\bf Construction of examples)} 
Let $S, S_m \subset \bP^2 \times \bP^1$ and $T \subset \bP^1 \times \bP^n$ be the hypersurfaces as above. 
Let $Y_1 := \bP^2 \times T=:Y_2$ and $D_1:= D_{ST}$ and $D_2:= D_{S_m T}$. Then the isomorphism $\phi_m \colon S \rightarrow S_m$ induces an isomorphism $\Phi_m \colon D_1 \rightarrow D_2$ via natural isomorphisms $ D_1 \simeq S \times_{\bP^1} T$ and $D_2 \simeq S_m \times_{\bP^1} T$. Hence we have a diagram 
\[
\xymatrix{
Y_1 & Y_2 \\
D_1 \ar[r]^{\simeq}_{\Phi_m} \ar@{^{(}->}[u] & D_2 \ar@{^{(}->}[u]  
}
\]
and define $Y_0:= Y_1 \cup^{\Phi_m} Y_2$. 
Then we see that $Y_0$ is an SNC CY variety since $D_i \in |{-}K_{Y_i}|$ and it is connected for $i=1,2$. Then $Y_0$ is not d-semistable by 
\begin{multline*}
\cN_{D_1 / Y_1} \otimes \Phi_m^* \cN_{D_2/Y_2} \\ 
\simeq \cO_{D_1} (p_S^*(3(H_S + \phi_m^* H_{S_m})+2(-K_S))) \simeq \cO_{D_1} (F_1+ \cdots + F_m + \Gamma_m), 
\end{multline*}
where $p_S \colon D_1 \rightarrow S$ is the projection and 
$F_i:= p_S^{-1}(f_i)$ for smooth $f_i \in |-K_S|$ for $i=1, \ldots m$ and $\Gamma_m:= p_S^{-1}(C_m)$ for smooth $C_m \in |3(H_S + \phi_m^* H_{S_m})+(2-m)(-K_S)|$. (One may see that the tensor product is ``quite positive'' since it decomposes to many divisors $F_1, \ldots, F_m, \Gamma_m$. This is why the 2nd Betti number of our examples $X(m)$ in the below can be large. ) 

Now let $\mu_1 \colon X_1 \rightarrow Y_1$ be the blow-up of $F_1, \ldots ,F_m$ and the strict transform of $\Gamma_m$ and $X_2:= Y_2$. 
Let $\tilde{D}_1 \subset X_1$ be the strict transform of $D_1 \subset Y_1$ with an isomorphism $\nu_1:= \mu_1|_{\tilde{D}_1}$. Then we have an isomorphism 
\[
\tilde{\Phi}_m := \Phi_m \circ \nu_1 \colon \tilde{D}_1 \rightarrow D_2
\] 
and let $$X_0(m):= X_1 \cup^{\tilde{\Phi}_m} X_2. $$
Then we see that $X_0(m)$ is an SNC CY variety by $\tilde{D}_1 \in |{-}K_{X_1}|$. We also see that $X_0(m)$ is d-semistable by the choice of the centers of the blow-up. 
By Theorem \ref{thm:smoothingSNC}, there exists a smoothing $\cX(m) \rightarrow \Delta^1$ of $X_0(m)$ and let $X(m)$ be its general fiber. 
Then we see that $X(m)$ is a Calabi-Yau manifold. This gives $X(m)$ as in Theorem \ref{thm:intro}. The properties can be checked as in \cite{MR4406696}. 
\end{eg}
 	
\section{Further examples and problems}	

\subsection{Further examples}
	
In \cite{Hashimoto:aa}, we constructed non-K\"{a}hler Calabi-Yau 3-folds by smoothing SNC CY 3-folds which are unions of two smooth ``quasi-Fano 3-folds'' glued along K3 surfaces. (More examples can be found in the earlier version of \cite{Hashimoto:aa} on arXiv.) 

We can also construct examples from SNC CY 3-folds whose intersection locus are abelian surfaces as Example \ref{eg:intersectionabelsurf}. Note that this example is new and such an example has not appeared before. 

\begin{eg}
We start with a degeneration of a projective CY 3-fold to an SNC CY 3-fold whose intersection locus is an abelian surface as follows. These should be well-known examples. 

Let $X=X_{(3,3)} \subset \bP^2 \times \bP^2$ be a general hypersurface of bi-degree $(3,3)$. 
Then it degenerates to an SNC hypersurface $Y_0:=E \times \bP^2 \cup \bP^2 \times E$, where $E \subset \bP^2$ is a cubic curve. 
Its intersection locus is the abelian surface $E \times E$. 
$Y_0$ is not d-semistable, but we may construct a d-semistable SNC CY 3-fold by blowing up suitable curves.  

Let $X:= X_{(3,0,1))} \cap X_{(0,3,1)} \subset \bP^2 \times \bP^2 \times \bP^1$ be the complete intersection of two hypersurfaces of bi-degrees $(3,0,1)$ and $(0,3,1)$ in $\bP^2 \times \bP^2 \times \bP^1$. This $X$ also has a fibration to $\bP^1$ whose general fiber is an abelian surface of product type. 

In these examples, abelian surfaces appeared as the intersection locus are of product type. 
It might be interesting to know whether simple abelian varieties can appear as such intersection locus, for example.   
\end{eg}

\begin{eg}\label{eg:intersectionabelsurf}
We can construct infinitely many topological types of non-K\"{a}hler CY 3-folds from SNC CY 3-folds whose singular locus are abelian surfaces as follows. This example is new to the author's knowledge. Note that we need the result of \cite{Chan:2019vv} since $H^1(\cO_{X_i}) \neq 0$ for the irreducible component $X_i$ in the following example.  

Let $E \subset \bP^2$ be a general cubic curve without complex multiplication. Let us fix a group structure on $E$ with a zero element $O \in E$. 
Let $f_i \subset E \times E$ be the fibers of the $i$-th projection $E \times E \rightarrow E$ for $i=1,2$ and $\Delta \subset E \times E$ be the diagonal. It is well known that the N\'{e}ron--Severi group $\NS(E \times E)$ is freely generated by the divisors $f_1, f_2, \Delta$.

 Let $a \in \bZ_{>1}$. 
Let $Y_1:= \bP^2 \times E$ and $Y_2 := E \times \bP^2$. 
Let $D_1 := E \times E \subset Y_1$ and $D_2:= E\times E \subset Y_2$. 
Let $$\phi_1 \colon E \times E \rightarrow E \times E ; (x,y) \mapsto (x-ay, y)$$ be the isomorphism induced  by the matrix $\begin{pmatrix}
1 & -a \\
0 & 1
\end{pmatrix}$ (use column vectors and $x-ay \in E$ is defined by the group structure of $E$). 
Similarly, let $\phi_2 \colon E \times E \rightarrow E\times E$ be the isomorphism induced  by the matrix $\begin{pmatrix}
1 & 0 \\
-a & 1
\end{pmatrix}$. Then we have the following diagram: 
\[
\xymatrix{
Y_1= \bP^2 \times E  & & E \times \bP^2 = Y_2 \\
D_1=E \times E \ar@{^{(}->}[u] &  E \times E \ar[l]_{\ \ \ \ \phi_1} \ar[r]^{\phi_2 \ \ \ \  } & E \times E = D_2 \ar@{^{(}->}[u] 
}. 
\]
Let $C_a:= \phi_1^* f_1$ and $D_a:= \phi_2^* f_2$. 
Then $C_a$ is numerically equivalent to the curve 
$\{ (ay, y) \mid y \in E \} \subset E \times E$. Then we compute that 
\[
C_a \cdot f_1 = a^2, \ C_a \cdot f_2 = 1, \ C_a \cdot \Delta= (a-1)^2. 
\]
From this, we obtain $C_a  \equiv (1-a) f_1 + (a^2 -a) f_2 + a \Delta$. 
Similarly, we compute $D_a \equiv (a^2-a) f_1 + (1-a) f_2 + a \Delta$.  By these and $ \deg \cO_{\bP^2}(3)|_E =9$, we obtain 
\[
\phi_1^* \cN_{D_1/Y_1} \otimes \phi_2^* \cN_{D_2/Y_2} \equiv 9( \phi_1^* (f_1) + \phi_2^* (f_2)) \equiv 9((a-1)^2(f_1+f_2) + 2a\Delta ). 
\]
By this, we see that the linear system $$|\phi_1^* \cN_{D_1/Y_1} \otimes \phi_2^* \cN_{D_2/Y_2} \otimes \cO_{E \times E} (- \Delta_1 - \cdots - \Delta_a)|$$ is ample and free, thus it contains a smooth member $\Gamma_a$. 

Let $\mu_1 \colon X_1 \rightarrow Y_1$ be the blow-up of smooth curves $\phi_1(\Delta_1), \ldots , \phi_1(\Delta_a)$ such that $\Delta_i \equiv \Delta$ for $i=1, \ldots, a$
and the strict transform of $\phi_1(\Gamma_a)$. Then $\mu_1$ induces an isomorphism $$\nu_1:= \mu_1|_{\tilde{D}_1} \colon \tilde{D}_1 \rightarrow D_1$$ from the strict transform $\tilde{D}_1 \subset X_1$ of $D_1$.  
Now let $$X_0:= X_0(a):= X_1 \cup^{\psi_a} Y_2$$ be an SNC variety 
which is a union of $X_1$ and $Y_2$ glued along the isomorphism 
\[
\psi_a:= \phi_2 \circ \phi_1^{-1} \circ \nu_1 \colon \tilde{D}_1 \rightarrow D_1 \rightarrow D_2. 
\]
Since $\tilde{D_1} \in |{-} K_{X_1}|$ and $D_2 \in |{-}K_{Y_2}|$, we see that $\omega_{X_0} \simeq \cO_{X_0}$. By the choice of the blow-up centers of $\mu_1$, we see that $X_0$ is d-semistable. 
Hence we can apply Theorem \ref{thm:smoothingSNC} to obtain a smoothing $\cX(a) \rightarrow \Delta^1$ and let $X(a)$ be its general fiber. 

\begin{prop} Let $X=X(a)$ be the 3-fold constructed as above. Then we have the following. 
\item[(i)] $X$ and $X_0$ are simply connected. 
\item[(ii)] $H^i(X, \cO_X) = 0 = H^0(X, \Omega^i_X)$ for $i=1,2$. 
Thus $X(a)$ is a Calabi-Yau 3-fold. 
\item[(iii)] The 2nd Betti number of $X$ is $b_2(X) = a+1$. 
\item[(iv)] The algebraic dimension $a(X)$ is $0$. 
\end{prop}

\begin{proof}
Recall that $X_1 \rightarrow \bP^2 \times E$ is the blow-up  of the $a+1$ smooth curves as above, $X_2 = E \times \bP^2$ and $X_{12} \simeq E \times E$. 

\noindent(i) This can be shown by a similar argument as in \cite[Proposition 3.10]{Hashimoto:aa} using the Clemens map $c \colon X \rightarrow X_0$. Recall that $c$ is a diffeomorphism over $X_0 \setminus X_{12}$, where $X_{12}:= X_1 \cap X_2 \simeq E\times E$. 
Indeed, by the Seifert--van Kampen theorem, we see that 
\[
\pi_1(X_0) \simeq \pi_1 (X_1) \ast_{\pi_1(X_{12})} \pi_1(X_2) \simeq \{ 1 \}. 
\]
Moreover, we see that 
\[
\pi_1(X) \simeq \pi_1(X_1') \ast_{\pi_1(\tilde{X}_{12})} \pi_1(X_2'), 
\]
where $X_i':= X_i \setminus X_{12}$ for $i=1,2$ and $\tilde{X}_{12}:= c^{-1}(X_{12})$ is an $S^1$-bundle over $X_{12}$. Then we can show that $\pi_1(X) \simeq \{1 \}$ as in \cite[Claim 3.12]{Hashimoto:aa}. 


\vspace{2mm}

\noindent(ii) Since we obtain $H^1(X, \bZ)=0$ by $\pi_1(X) = \{1 \}$, we see that $H^1(X, \cO_{X}) =0 = H^0(X, \Omega^1_X)$ by the $E_1$-degeneration of the Hodge to de Rham spectral sequence on $X$ (cf. \cite[Remark 3.8]{Hashimoto:aa}). By Serre duality, we see that $H^2(X, \cO_X) =0$. We see that $\Pic (X) \simeq H^2(X, \bZ)$ by the exponential sequence. This implies that $H^0(X, \Omega^2_X) =0$. 

\vspace{2mm}

\noindent(iii) We have an exact sequence 
\[
0 \rightarrow H^2(X_0, \bZ) \rightarrow H^2 (X_1, \bZ) \oplus H^2(X_2, \bZ) \xrightarrow{\alpha} H^2(X_{12}, \bZ). 
\]
Note that the injectivity on the left follows from the surjectivity of the homomorphism $H^1(X_1, \bZ) \oplus H^1(X_2, \bZ) \rightarrow H^1(X_{12}, \bZ)$ by explicit calculation.  
Since $\Image \alpha \simeq \bZ^3$, $H^2(X_1, \bZ) \simeq \bZ^{2+a+1}$ and $H^2 (X_2, \bZ) \simeq \bZ^{2}$, we see that $H^2(X_0, \bZ) \simeq \bZ^{a+2}$. 
By the standard argument as in \cite[2.6]{MR4085665}, we obtain an exact sequence 
\begin{equation}\label{eq:CSseq}
H^1(X, \bC) \rightarrow H^0(X_{12}, \bC) \rightarrow H^2(X_0, \bC) \rightarrow H^2(X, \bC) \rightarrow 0
\end{equation}
 and see that $H^2(X,\bC) \simeq \bC^{a+1}$ by this sequence. Hence we obtain $b_2(X) = a+1$. 

\vspace{2mm}

\noindent(iv) Note that $H^2(X_0, \bZ) \simeq H^2(\cX, \bZ) \simeq \Pic \cX$ and $H^2(X, \bZ) \simeq \Pic X$ by $H^i(X, \cO_X) =0 = H^i(X_0, \cO_{X_0})$ for $i=1,2$. Hence we have a surjection $\Pic \cX \rightarrow \Pic X$ by the sequence (\ref{eq:CSseq}). 

Suppose that there exists a line bundle $\cL_t \in \Pic X$ with $\kappa (\cL_t) \ge 1$. Then there exists $\cL_0 \in \Pic X_0$ such that $\kappa(\cL_0) \ge 1$ and $\cL_0|_{X_i}$ is effective for $i=1,2$ by the same argument as in \cite[Proposition 3.19 (iii)]{Hashimoto:aa}. This contradicts the following claim. 

\begin{claim}\label{claim:effective}
Let $\cL_0 \in \Pic X_0$ be a line bundle such that $\cL_i := \cL_0|_{X_i}$ is effective for $i=1,2$. Let $F_1, \ldots ,F_{a+1} \subset X_1$ be the exceptional divisors over $\phi_1(\Delta_1), \ldots , \phi_1(\Delta_a)$ and $\phi_1(\Gamma_a)$. Write 
\[
\cL_1 = \mu_1^*(\cO_{\bP^2}(\alpha) \boxtimes \cO_E(D_1)) \otimes \cO_{X_1} \left( \sum_{j=1}^{a+1} b_j F_j \right), \ \ 
\cL_2 = \cO_E(D_2) \boxtimes \cO_{\bP^2}(\alpha') 
\]
for some integers $\alpha, \alpha', b_1, \ldots , b_{a+1}$ and let $d_i:= \deg D_i$ for $i=1,2$. 
(For varieties $Z_i$ and $\cL_i \in \Pic Z_i$ for $i=1,2$, let $\cL_1 \boxtimes \cL_2:= \pi_1^* \cL_1 \otimes \pi_2^* \cL_2 \in \Pic (Z_1 \times Z_2)$ for the projections $\pi_i \colon Z_1 \times Z_2 \rightarrow Z_i$.  )

Then we see that $\alpha = \alpha'=0$ and $d_1=d_2 =0$. 
In particular, $\kappa(\cL_0) = 0$. 
\end{claim}

\begin{proof}[Proof of Claim]
Since $\cL_i$ is effective, we see that $\alpha, \alpha', d_1, d_2 \ge 0$. Suppose that $(\alpha, d_1) \neq (0,0)$. Then we have $(\alpha', d_2) \neq (0,0)$. We shall see a contradiction in the following. 
We use the identification $$\NS (E \times E) \simeq \bZ^3 ; \ \  a_1 f_1+ a_2 f_2 + a_3 \Delta \mapsto (a_1, a_2, a_3). $$
By the isomorphism $\cL_1|_{X_{12}} \simeq \cL_2|_{X_{12}}$, 
we see that 
\[
3 \alpha \cdot C_a + p_2^*(D_1) + \sum_{j=1}^a b_j \Delta + b_{a+1} \Gamma_a \equiv p_1^*(D_2) + 3\alpha' \cdot D_a, 
\]
where $p_i \colon E \times E \rightarrow E$ is the $i$-th projection for $i=1,2$. 
By the above identification, this corresponds to the equality 
\begin{multline*}
3 \alpha (1-a, a-a^2,a) + (0,d_1,0) + \sum_{j=1}^a b_j (0, 0, 1) + 9b_{a+1} \left( (a-1)^2, (a-1)^2, a \right) \\ 
= (d_2, 0,0) + 3\alpha' (a^2-a, 1-a, a) \in \bZ^3. 
\end{multline*}
The equality on the 1st and 2nd coordinates implies the equality 
\begin{multline*}
\left( 3\alpha(1-a), 3\alpha(a^2-a) +d_1 \right) + 9b_{a+1} \left( (a-1)^2, (a-1)^2 \right) \\ 
= \left( d_2+3\alpha' (a^2-a), 3\alpha' (1-a) \right) 
\end{multline*}
and we see that 
\[
3\alpha (1-a) - (d_2 + 3 \alpha'(a^2-a)) = (3\alpha(a^2-a) + d_1) - 3\alpha'(1-a). 
\]
This is a contradiction since the L.H.S.\ is negative and the R.H.S.\ is positive. This finishes the proof of Claim \ref{claim:effective}
\end{proof}

By Claim \ref{claim:effective}, we see that there does not exist $\cL_t \in \Pic X$ such that $\kappa (\cL_t) \ge 1$. 
Hence we see that $a(X)=0$.
\end{proof}
\end{eg}

\subsection{Problems} 
Let $X_0 = X_1 \cup X_2$ be an SNC CY variety with two smooth irreducible components $X_1$ and $X_2$. 
Degenerations of a smooth CY manifold to such an SNC CY variety 
are called {\it Tyurin degenerations} although it is often assumed that $X_1, X_2$ are ``quasi-Fano'' varieties and $X_1 \cap X_2$ is a (strict) CY manifold (cf. \cite[5.4]{MR2708981}, \cite{MR3751815}). In Example \ref{eg:intersectionabelsurf}, we treated CY 3-folds which admit degenerations to SNC CY 3-folds of the form $X_1 \cup X_2$ such that $X_1 \cap X_2$ are abelian surfaces.   
One can ask which kind of varieties can appear as the intersection locus of such SNC Calabi-Yau varieties. 

\begin{eg}
Let $X \subset \bP^n \times \bP^m$ be a smooth hypersurface of bi-degree $(n+1, m+1)$. Then it degenerates to an SNC CY variety 
$X_0=\bP^n \times X_{m+1} \cup X_{n+1} \times \bP^m$.  The intersection locus of $X_0$ is $X_{n+1} \times X_{m+1}$, thus a product of two Calabi-Yau manifolds can appear as the intersection locus. 
\end{eg}

%

Tyurin \cite{MR2112600} called a CY 3-fold with Tyurin degenerations a {\it constructive}  CY 3-fold. He speculated that every CY 3-fold is constructive. 
(A referee pointed out that the mirror quintic 3-fold has very few deformations and none of these is a Tyurin degeneration (cf. \cite{MR2282973}, \cite{MR3751815}). )

It is also written that topological types of constructive CY 3-folds should be finite. Although there are infinitely many topological types of non-K\"{a}hler Calabi-Yau manifolds with Tyurin degenerations as in examples in Section \ref{section:CYnexamples}, it is still not clear in the projective case.  

It would also be interesting to see bimeromorphic relations between our examples of non-K\"{a}hler CY manifolds (cf. \cite{MR909231}). 
	
\section*{Acknowledgement}
The author is grateful to Kenji Hashimoto for useful discussions. He thanks Alexander Kasprzyk for the opportunity of the seminar talk and the proceedings. 
He would also like to thank the referees for constructive comments. 
This work was partially supported by JSPS KAKENHI Grant Numbers JP17H06127, JP19K14509.

\bibliographystyle{amsalpha}
\bibliography{sanobibs-bddlogcy}

\end{document}